\documentclass{amsart}[12pt]
\usepackage{amsmath, amsthm, amscd, amsfonts,}
\usepackage{graphicx,float}

\usepackage{amssymb}

\usepackage{pb-diagram}
\usepackage{lamsarrow}
\usepackage{pb-lams}

\newcommand{\PP}{{\mathbb{P}}}

\DeclareMathOperator{\Add}{Add}

\def\k{\kappa}

\def\l{\lambda}

\def\a{\alpha}

\def\b{\beta}

\newtheorem{theorem}{Theorem}[section]
\newtheorem{lemma}[theorem]{Lemma}

\newtheorem{definition}[theorem]{Definition}

\newtheorem{remark}[theorem]{Remark}

\numberwithin{equation}{section}

\usepackage{pb-diagram}

\def\l{\lambda}

\def\rmark{\mbox{$\rm\bf\rule{0.06em}{1.45ex}\kern-0.05em R$}}
\def\pmark{\mbox{$\rm\bf\rule{0.06em}{1.45ex}\kern-0.05em P$}}
\def\nmark{\mbox{$\rm\bf\rule{0.06em}{1.45ex}\kern-0.05em N$}}
\def\vdash{\mbox{$\rm\| \kern-0.13em -$}}

\def\l{\lambda}

\def\rmark{\mbox{$\rm\bf\rule{0.06em}{1.45ex}\kern-0.05em R$}}
\def\pmark{\mbox{$\rm\bf\rule{0.06em}{1.45ex}\kern-0.05em P$}}
\def\nmark{\mbox{$\rm\bf\rule{0.06em}{1.45ex}\kern-0.05em N$}}
\def\vdash{\mbox{$\rm\| \kern-0.13em -$}}

\begin{document}

\title[On the notion of generic cut for models of $ZFC$]{On the notion of generic cut for models of $ZFC$}

\author[Mohammad Golshani ]{Mohammad
  Golshani }

\thanks{The author's research has been supported by a grant from IPM (No. 91030417).} \maketitle



\begin{abstract}
We define the notion of generic cut between models of $ZFC$ and give some examples.
\end{abstract}

\section{introduction}
Given models  $V  \subseteq W$ of $ZFC$, we define the notion of generic cut of $(V, W)$,  and prove some results about it.  We usually assume that $W$ is a generic extension of $V$ by a set or a class forcing notion, but some of our results work for general cases.

\section{generic cut for pairs of models of $ZFC$}
Let's start with the main definition.
\begin{definition}
Suppose that $V  \subseteq W$ are models of $ZFC$ with the same ordinals and $\a, \b$ are ordinals.
 An $(\a, \b)-$generic cut of $(V, W)$,  is a pair $\langle \vec{V}, \vec{W} \rangle,$ where
\begin{enumerate}
\item $\vec{V}=\langle V_i: i<\a \rangle$ is a $\subset$-increasing chain of generic extensions of $V$, with $V_0=V,$
\item $\vec{W}=\langle W_j: j<\b \rangle$ is a $\subset$-decreasing chain of grounds \footnote{Recall that $V$ is a ground of $W$, if $W$ is a set generic extension of $V$ by some forcing notion in $V$.} of $W$, with $W_0=W,$
\item $\forall i<\a, j<\b, V_i \subset W_j,$
\item There is no inner model $V \subseteq M \subseteq W$ of $ZFC$ such that   $\forall i<\a, j<\b, V_i \subset M \subset W_j.$
\end{enumerate}
\end{definition}
Note that if $W$ is a set generic extension of $V$, then any $V_i$ is a ground of $W$ and each $W_j$ is a generic extension of $V$.
\begin{lemma}
Suppose that $V  \subseteq W$ are models of $ZFC$ and there exists an  $(\a, \b)$-generic cut of $(V, W).$ Then $\a \leq [W:V]^U$ and $\b \leq [W:V]_D,$ where $[W:V]^U, [W:V]_D$ are defined as in \cite{golshani}.
\end{lemma}
The next theorem shows that there are no generic cuts in the extension by Cohen forcing.
\begin{theorem}
Let $\mathbb{P}=\Add(\omega, 1)$ be the Cohen forcing for adding a new Cohen real and let $G$ be $\mathbb{P}-$generic over $V$.Then there exists no $(\a, \b)$-generic cut of $(V, V[G])$.
\end{theorem}
\begin{proof}
Assume towards a contradiction  that $(\vec{V}, \vec{W})$ witnesses an $(\a, \b)$-generic cut  of $(V, V[G])$. We consider several cases:
\begin{enumerate}
\item At least one of $\a$ or $\b$ is uncountable.
 Suppose for example that $\a \geq \aleph_1.$ It follows that $\langle V_i: i< \aleph_1 \rangle$  is  a $\subset-$increasing chain of generic extensions of $V$, which are included in $V[G]$, which contradicts \cite{golshani} Theorem 2.5$(3)$. So from now on we assume that both $\a, \b$ are countable.

\item Both of $\a=\a^{-}+1$ and $\b=\b^{-}+1$ are successor ordinals. Then we have $V \subseteq V_{\a^{-}} \subset W_{\b^{-}} \subseteq V[G],$ and there are no inner models $M$ of $V[G]$ with $V_{\a^{-}} \subset M \subset W_{\b^{-}},$ which is clearly impossible (in fact there should be $2^{\aleph_0}$ such $M$'s).

\item Both of $\a, \b < \aleph_1$ are limit ordinals.
We can imagine each $V_i, i<\a,$ is of the form $V_i=V[a_i],$ for some Cohen real $a_i$ and similarly each $W_j, j<\b,$ is of the form $W_j=V[b_j],$ for some Cohen real $b_j.$ Let $a\in V[G]$ be a Cohen real over $V$, coding all of $a_i$'s, $i<\a.$ Then for all $i<\a, j<\b, V_i \subset V[a] \subset W_j,$ a contradiction.

\item One of $\a$ or $\b$ is a limit ordinal $>0$ and the other one is a successor ordinal. Let's assume that $\a$ is a limit ordinal and $\b=\b^{-}+1$ is a successor ordinal. As $cf(\a)=\omega,$ we can just consider the case where $\a=\omega.$ Then for all $i<\omega$ we have $V_i \subset W_{\b^{-}}.$ We can assume that each $V_i$ is of the form $V_i=V[a_i],$ for some Cohen real $a_i,$ and that $W_{\b^{-}}=V[b],$ for some Cohen real $b$. Using a fix bijection $f: \omega \leftrightarrow \omega\times \omega, f\in V,$ we can imagine $b$ as an $\omega$-sequence $\langle b_i: i<\omega \rangle$ of reals which is $\Add(\omega, \omega)$-generic over $V$, so that $W_{\b^{-}}=V[\langle b_i: i<\omega \rangle].$ We can further suppose that each $b_i$ codes $a_i$ (i.e., $a_i\in V[b_i]$) Let us now define a new sequence $\langle c_i: i<\omega \rangle$ of reals in $V[\langle b_i: i<\omega \rangle],$ so that
    $c_i(0)=0,$ and $c_i \upharpoonright [1, \omega)=b_i \upharpoonright [1, \omega).$ Finally let $M=V[\langle c_i: i<\omega \rangle].$ It is clear that each $V_i \subset M.$ But also $M \subset W_{\b^{-}},$ as the real $t\in W_{\b^{-}}$ defined by $t(i)=b_i(0)$  is not in $M$ (by a genericity argument). We get a contradiction.
\item One of $\a$ or $\b$ is $0$ and the other one is a limit or a successor ordinal. Then as above we can get a contradiction.
\end{enumerate}
\end{proof}
\begin{theorem}
Assume $\a, \b$ are ordinals. Then
 there exists a generic extension $V[G]$ of $V$, such that there is an $(\a+1, \b+1)-$generic cut of $(V, V[G]),$
\end{theorem}
\begin{proof}
 Let $\PP_1=\Add(\omega, \a),$ and let $G_1=\langle a_i: i<\a \rangle$ be a generic filter over $V$. Force over $V[G_1]$ by any forcing notion which produces a minimal extension $V[G_1][G_2]$ of $V[G_1]$. Finally force over $V[G_1][G_2]$ by $\PP_3=\Add(\omega, \b),$ and let $G_3=\langle b_j: j<\b \rangle$ be a generic filter over $V[G_1][G_2]$. Let
\begin{itemize}
\item $\vec{V}=\langle V_i: i\leq \a    \rangle$, where $V_i=V[\langle a_\xi: \xi< i \rangle]$ for $i<\a,$ and $V_\a=V[G_1],$
\item $\vec{W}=\langle  W_j: j\leq \b \rangle,$ where $W_j=V[G_1][G_2][\langle b_\xi: j \leq \xi <\b \rangle]$ for $j<\b,$ and $W_\b=V[G_1][G_2]$.
\end{itemize}
Then $(\vec{V}, \vec{W})$ witnesses an $(\a+1, \b+1)-$generic cut of $(V, V[G]).$
\end{proof}
\begin{remark}
If $V$ satisfies $GCH$, then we can find $V[G],$ so that it also satisfies the $GCH$; it suffices to work with  $\Add(|\a|^+, \a)$ and $\Add(|\b|^+, \b)$ instead of  $\Add(\omega, \a)$ and $\Add(\omega, \b)$  respectively.
\end{remark}

\begin{theorem}
Assume $\l_1, \l_2$ are infinite regular cardinals and $\k$ is a measurable cardinal above them. Then in a generic extension $V[G]$ of $V$, there exists a $(\l_1, \l_2)-$generic cut of $(V, V[G]).$
\end{theorem}
\begin{proof}
By \cite{spa}, we can find a generic extension $V[G_1]$ of $V$, by a forcing of size $<\k,$ such that in $V[G_1],$ there exists a $(\l_1, \l_2)$-gap of $P(\omega)/fin$. $\k$ remains measurable in $V[G_1],$ so let $U$ be a normal measure on $\k$ in $V[G_1],$ and force with the corresponding Prikry forcing $\PP_U,$ and let $G_2$ be $\PP_U$-generic over $V[G_1].$ Let $V[G]=V[G_1][G_2].$ By \cite{gitik},
\begin{center}
$(\{M: M$ is a model of $ZFC, V[G_1] \subseteq M \subseteq V[G] \}, \subseteq) \cong (P(\omega)/fin, \subseteq^*).$
\end{center}
Now the result should be clear, as a $(\l_1, \l_2)$-gap in $P(\omega)/fin$, produces the corresponding  $(\l_1, \l_2)-$generic cut $(\vec{V}, \vec{W})$ of $(V[G_1], V[G]),$ which in turn produces the same $(\l_1, \l_2)-$generic cut of $(V, V[G])$ (by adding $V$ at the beginning of $\vec{V}$).
\end{proof}

School of Mathematics, Institute for Research in Fundamental Sciences (IPM), P.O. Box:
19395-5746, Tehran-Iran.

E-mail address: golshani.m@gmail.com


\begin{thebibliography}{99}

\bibitem{gitik} Gitik, Moti, Kanovei, Vladimir; Koepke, Peter; Intermediate submodels of Prikry generic extensions, preprint.

\bibitem{golshani} Golshani, Mohammad; On the notions of dimension and transcendence degree for models of $ZFC$, preprint.

\bibitem{spa} Spasojevi\'{c}, Zoran; Some results on gaps, Topology Appl. 56 (1994), no. 2, 129--139.
\end{thebibliography}
\end{document}